\documentclass[12pt]{article}
\usepackage[utf8]{}
\usepackage[normalem]{ulem} 
\usepackage{indentfirst}
\usepackage{amsmath,bm}
\usepackage{amsfonts}
\usepackage{mathtools}
\usepackage{amsthm}
\usepackage{amssymb} 
\usepackage{amsmath}
\usepackage{verbatim}
\usepackage{enumerate}
\usepackage{enumitem}
\usepackage{setspace}
\usepackage[top=1in, bottom=1in, left=1in, right=1in]{geometry}

\usepackage{float}
\usepackage{algorithm}
\usepackage{color, soul, mathabx}
\usepackage{graphicx}
\usepackage{subfigure}
\usepackage{theoremref}
\usepackage[square,sort,comma,numbers]{natbib}
\usepackage{authblk}
\usepackage{url}
\usepackage[font=footnotesize]{caption}
\usepackage{ulem}
\usepackage{bbm}

\newcommand{\bbeta}{\boldsymbol{\beta}}
\newcommand{\bpi}{\boldsymbol{\pi}}
\newcommand{\bx}{\mathbf{x}}
\newcommand{\ba}{\mathbf{a}}
\newcommand{\bB}{\mathbf{B}}
\newcommand{\bzero}{\mathbf{0}}
\newcommand{\cB}{\mathcal{B}}

\newcommand{\barB}{\bar{\bB}}
\newcommand{\opt}{\mathrm{opt}}
\newcommand{\slj}{\sum\limits_{j = 1}}

\providecommand{\keywords}[1]
{
  \small	
  \textbf{\textit{Keywords---}} #1
}

\renewenvironment{abstract}
{\begin{quote}
\noindent \rule{\linewidth}{.5pt}\par{\bfseries \abstractname.}}
{\medskip\noindent \rule{\linewidth}{.5pt}
\end{quote}
}

\newtheorem{defn}{Definition}
\newtheorem{myl}{Lemma}
\newtheorem{mypro}{Proposition}
\newtheorem{theorem}{Theorem}
\newtheorem{cor}{Corollary}
\newtheorem{remark}{Remark}
\newtheorem{example}[theorem]{Example}
\newtheorem*{excont}{Example \continuation}
\newcommand{\continuation}{??}
\newenvironment{continueexample}[1]
 {\renewcommand{\continuation}{\ref{#1}}\excont[continued]}
 {\endexcont}

\title{Primal Construction of Integer Programming Value Functions}

\def\tajayi#1{\color{black}#1}

\def\wenxin#1{\color{black}#1}

\newenvironment{makeshiftResult}[2]
{
\noindent\textbf{#1~#2.}\em}
{

}

\title{A Gilmore-Gomory Construction of Integer Programming Value Functions}

\author[1]{\normalsize Seth Brown}
\author[2]{\normalsize Wenxin Zhang}
\author[1,3]{\normalsize Temitayo Ajayi}
\author[1]{\normalsize Andrew J. Schaefer \thanks{andrew.schaefer@rice.edu}}

\affil[1]{\footnotesize Rice University, Department of Computational and Applied Mathematics}
\affil[2]{\footnotesize Tsinghua University, Department of Industrial Engineering}
\affil[3]{\footnotesize The University of Texas MD Anderson Cancer Center, Department of Radiation Oncology}

\date{}

\begin{document}

\maketitle

\begin{abstract}
\small
In this paper, we analyze how sequentially introducing decision variables into an integer program (IP) affects the value function and its level sets. We use a Gilmore-Gomory approach to find parametrized IP value functions over a restricted set of variables. We introduce the notion of maximal connected \color{black}subsets of level sets - volumes in which changes to the constraint right-hand side have no effect on the value function - and relate these structures to IP value functions and optimal solutions.\par
\end{abstract}

\section{Introduction}
Given a constraint matrix $\mathbf{A} \in \mathbb{Z}^{m \times n}_{+}$ and objective coefficients $\mathbf{c} \in \mathbb{Z}^{n}$, \color{black}the integer programming (IP) value function represents the optimal objective value of an IP parametrized by the right-hand side. Let $\mathbf{b} \in \mathbb{Z}^{m}_{+}$ be a component-wise upper bound on permissible right-hand sides. Define $\cB \coloneqq \bigtimes\limits_{i = 1}^{m} [0, b_{i}]$\color{black}. Given $\bbeta \in\cB$\color{black}, the parametrized IP, IP($\bbeta$), and its value function, $z:\cB\to \mathbb{R}$, are defined by 
\begin{align*}
z(\bbeta) \coloneqq\max\limits_{\bx}\left\{\mathbf{c}^{\top}\bx : \mathbf{A}\bx \leq \bbeta,\ 
\bx \in \mathbb{Z}^{n}_{+}\right\}.\tag{IP($\bbeta$)}
\end{align*}
Note that because of the nonnegativity of $\mathbf{A}$ and $\mathbf{b}$, without loss of generality we assume $\mathbf{c}$ is also nonnegative. Denote the $j^{th}$ column vector of the matrix $\mathbf{A}$ by $\mathbf{a}_j$. We assume that $\ba_{j} \neq \bzero$ and $\ba_{j} \leq \mathbf{b}$, for all $j \in \{1,\dots,n\}$. Note that $z(\bbeta) = z(\lfloor \bbeta \rfloor)$ for all $\bbeta \in \cB$, and \color{black} that each parametrized IP is feasible because $\bzero$ is a feasible solution for each $\bbeta \in \mathcal{B}$, and the IPs are bounded because $\mathbf{A}$ is nonnegative with no zero columns.\par
Studying the value function of the parametrized IP is particularly useful in cases where IP($\bbeta$) must be solved many times for different right-hand sides, such as in bilevel programming, where IP($\bbeta$) in the form of the follower problem has its right-hand side depend on the leader problem decisions. Similarly, in stochastic programming, IP($\bbeta$) in the form of the second-stage problem is dependent on both the first-stage decisions and the resolution of uncertain values. Parametrized solution approaches to these problems can be found in, e.g., \cite{ls2017,wx2017} for bilevel optimization, and \cite{ss1998,Ahmed2003,Kong2006} for stochastic optimization.\color{black}\par
Early work on parametrized IP focuses on value functions and their construction. \color{black}\cite{Blair77} provide bounds on the variation of the value function relative to the variation of the right-hand sides for both mixed-integer programs (MIPs) and IPs.
\cite{Blair82} show that IP value functions are Gomory functions, and \cite{Blair95} generalizes this result for MIPs. \cite{Llewellyn1993} use the connection between IP value functions and Gomory functions to produce a primal-dual algorithm for solving 0-1 IPs.
\cite{Williams1996} provides a doubly recursive procedure using Chv\'{a}tal functions to construct the value function of a conic integer linear program.

More recent work focuses on the construction and applications of value functions. \cite{Ralphs2014} provide an algorithm for constructing value functions for MIPs. Value functions have also been incorporated into solution approaches for two-stage stochastic integer programs \citep{Ozaltin2012,Trapp2015}.
Various value function approaches have also been applied to the solution of mixed-integer bilevel programs.
\cite{Onur2019} use a generalized MIP value function, which takes both the objective function coefficients and the constraint right-hand side as arguments, to solve both stochastic and multifollower bilevel MIPs. \cite{Basu2018} use Gomory, Chv\'{a}tal, and Jeroslow functions, as well as value functions, to analyze the representability of mixed-integer bilevel programs.

Some of our results are analogous to the Gilmore-Gomory approach for knapsack problems \cite{gg1966}, which uses dynamic programming to recursively determine the value function of a one- or two-dimensional knapsack IP. Our method focuses on one variable at each step; in contrast, the approach in \cite{gg1966} considers all variables at each recursive level.

Thus far, level sets of the value function, along with related topics, such as level-set-minimal vectors, have been sparsely studied. \cite{Kong2006} use minimal tenders in a stochastic programming solution approach. \cite{Trapp2014} introduce the notion of level-set-minimal vectors and use them to represent the value functions of the first and second stage of a stochastic program. \cite{Ajayi2020} relate level-set-minimal vectors to the linear programming relaxation gap functions. The stability regions discussed in \cite{kk2009} are similar to the value function level sets discussed in our work, but \cite{kk2009} use a parametrization of the objective function to study changes to optimal solutions, not the objective value.\color{black}

In this paper, we examine the properties of the level sets of IP value functions in detail, especially the connections among these level sets as primal variables are added to a problem.
We develop a primal construction of the IP value function, in constrast to dual approaches, for example, using Chv\'{a}tal or Gomory functions \citep{Blair82}. A primal approach enables us to analyze the behavior of restricted versions of the IP. Our contributions are as follows:
\begin{enumerate}
\item We analyze the IP value function over subsets of primal variables, including how the value function's level sets change as primal variables are added iteratively to the formulation in a Gilmore-Gomory-type procedure.
\item We characterize the structure of level sets of IPs. We show how the level sets of a restricted IP relate to the level sets when a new variable is included.
\item We introduce the notion of maximal connected \color{black}subsets of level sets and demonstrate several properties of these subsets. These results can be used to determine common optimal solutions within a maximal connected subset. 
\end{enumerate}
Our key results include:
\begin{itemize}\item A sufficient condition for a right-hand side to be level-set minimal (\thref{down});
\item A recursive approach to construct variable-restricted value functions (\thref{genstepup});
\item Additional connectedness properties of maximally-connected level sets and lattices (\thref{Csteppro} and \thref{MC_Adjacent}).\end{itemize}\color{black}

\section{Level Sets of the IP Value Function}\label{sec:Characterization}
For any $\alpha \in \mathbb{R}$, the \textit{level set} $S(\alpha)$ of the value function $z$ is the set of right-hand sides over which the function takes on the value $\alpha$: $S(\alpha)\coloneqq\left\{\bm{\beta}\in\mathcal{B}\,\vert\,z(\bm{\beta})=\alpha\right\}$. If $S(\alpha )=\emptyset$, then for all $\bbeta \in \mathcal{B}$, $z(\bm{\beta}) \neq \alpha$. In particular, under our assumptions, for all $\alpha\in (-\infty,0)\cup(z(\mathbf{b}),\infty)$, $S(\alpha) =  
\emptyset$. We examine the structure of level sets of IPs and develop properties of level sets over subsets of the primal variables. We focus on the case of adding or removing a single primal variable at a time, but these results can also be extended to sets of primal variables.

Define the restricted value function, $z_k(\bbeta)$, with respect to the parametrized IP over the first $k\in\{1, ..., n\}$ variables, IP$_k(\bm{\beta})$, as
\begin{align*}
z_k(\bbeta) \coloneqq \max\limits_{\bx}\left\{\sum_{j=1}^kc_jx_j : \sum_{j=1}^k\mathbf{a}_jx_j \leq \bbeta,\ 
\bx \in \mathbb{Z}^{k}_{+}\right\}.\tag{IP$_k(\bbeta$)}
\end{align*}
Define
    $S_k(\alpha)\coloneqq\left\{\bm{\beta}\in\mathcal{B}\,\vert\,z_k(\bm{\beta})=\alpha\right\}$ and
\wenxin{
    $\opt_k(\bm{\beta})\coloneqq\arg\max\limits_{\mathbf{x}}\left\{\sum_{j=1}^kc_jx_j: \sum_{j=1}^k\mathbf{a}_jx_j \leq \bbeta,\ 
\bx \in \mathbb{Z}^{k}_{+}\right\}$.
}  \color{black}

\subsection{Properties of Restricted Value Functions}
We apply fundamental properties of the IP value function to restricted value functions. 

\begin{mypro}\cite{wolsey1981integer}\thlabel{elementary} Basic properties of the restricted value functions $z_{k}$, for $k = 1,\dots,n,$ include:
\begin{enumerate}[label = (\arabic*)]
\item $z_{k}(\mathbf{a}_j)\geq c_j$ for $j=1,\dots,k$.
\item  $z_{k}$ is nondecreasing over $\mathcal{B}$.
\item $z_{k}$ is superadditive over $\mathcal{B}$; i.e., for all $\bbeta_1$, $\bbeta_2\in \mathcal{B}$, $\bbeta_1+\bbeta_2\in \mathcal{B}$, $z_{k}(\bbeta_1)+z_{k}(\bbeta_2)\leq z_{k}(\bbeta_1+\bbeta_2)$. 
\end{enumerate}
\end{mypro}
\begin{mypro}\cite{NemhauserWolsey1988}\thlabel{IPCS} Given $k\in \{1, ..., n\}$, if \wenxin{ $\mathbf{x}^*$}\color{black}\ $\in \opt_k(\bm{\beta})$, then for all $\mathbf{x} \in \mathbb{Z}_+^k$ such that $\mathbf{x} \leq $ \wenxin{$\mathbf{x}^*$, }\color{black} $z_k(\sum_{j=1}^k\mathbf{a}_jx_j) = \sum_{j=1}^k c_jx_j$ and $z_k(\sum_{j=1}^k\mathbf{a}_jx_j)+z_k(\bbeta-\sum_{j=1}^k\mathbf{a}_jx_j) = z_k(\bbeta)$.
\end{mypro}
\thref{IPCS} is often referred to as IP complementary slackness.

\begin{myl}\thlabel{nondecreasingbetweeniterations}
For all $k \in \{2,...,n\}$, if $\bm{\beta} \in S_{k-1}(\alpha_1)$ and $\bm{\beta} \in S_k(\alpha_2)$, then $\alpha_2\geq \alpha_1$.
\end{myl}

\thref{nondecreasingbetweeniterations} states that $z_k(\bbeta)$ increases monotonically with $k$ for fixed $\bbeta$ - as more variables become available, the value of the restricted problem can only improve.

\begin{myl}\thlabel{notOptimal}
Given $k \in \{2, ..., n\}$, $\mathbf{x}\in\mathbb{Z}^k_+$ such that $\sum_{j=1}^k\mathbf{a}_j x_j\leq \mathbf{b}$, let $\alpha=\sum_{j=1}^kc_j x_j$. If $S_{k}(\alpha)= \emptyset$, then ${\mathbf{x}}\notin\opt_k(\sum_{j=1}^k\mathbf{a}_j{x}_j)$.
\end{myl}
Hence, \thref{notOptimal} gives a necessary condition for optimality. Let $\mathbf{e}_i$ indicate the $i$th unit vector.

\begin{mypro}\thlabel{schaeferoriginal}
Given $k \in \{2, ..., n\}$ and $\bbeta \in \cB$, let $\mathbf{x}^*\in \opt_k(\bm{\beta})$. We have:
\begin{enumerate}[label = (\arabic*)]
\item  For all $t \in \mathbb{Z}_{+}$ such that $t \leq x^{*}_{k}$, $z_k(\bm{\beta})=z_k(\bm{\beta}-t\mathbf{a}_{k})+z_k(t\mathbf{a}_k)=z_{k-1}(\bm{\beta}-\mathbf{a}_{k}x^*_{k})+c_k{x}_k^*$. Further, $\bx^{*} - t\mathbf{e}_{k} \in \opt_k(\bbeta - t\ba_{k})$, and if $\bbeta \in S_k(\alpha)$, then $\bbeta - t\ba_{k} \in S_{k}(\alpha - tc_{k})$. \label{orig.1}
\item If $\bm{\beta}\in S_k(\alpha)$, then $\bm{\beta}-\mathbf{a}_{k}x^*_{k}\in S_{k-1}(\alpha-c_{k}x^*_{k})$. \label{orig.3}
\item If $\bbeta \in S_{k}(\alpha)$ and $S_{k-1}(\alpha) = \emptyset$, then $x^{*}_{k} > 0$. \label{orig.4}
\end{enumerate}
\end{mypro}
\begin{proof}
Statements \ref{orig.1}-\ref{orig.3} follow from \thref{IPCS}. For \ref{orig.4}, suppose that $x^{*}_{k} = 0$. By \ref{orig.3}, $\bm{\beta}\in S_{k-1}(\alpha)$, which is a contradiction.
\end{proof}
\thref{schaeferoriginal} demonstrates how the structure of optimal solutions can determine members of level sets over a restricted set of primal variables. In particular, \thref{schaeferoriginal} shows the impact on the value of a particular $\bbeta$ after removing a primal variable from the restricted problem. In addition, given $\bbeta$ and $k$, and given $\mathbf{x}^* \in \opt_k(\bbeta)$, for any feasible solution $ \mathbf{x}\leq \mathbf{x}^*$ \thref{IPCS} implies ${\mathbf{x}} \in \opt_k(\sum_{j=1}^k \mathbf{a}_jx_j)$.

\subsection{Level-Set-Minimal Vectors}
Level-set-minimal vectors represent the efficient frontiers for the corresponding level sets and can be used to construct the boundaries of a value function's level sets.
\begin{defn}\cite{Trapp2014}\thlabel{minimal}
A vector $\bm{\beta}\in\mathcal{B}$ is \textbf{level-set-minimal} with respect to $z:\mathcal{B}\rightarrow \mathbb{R}$ if $z(\bm{\beta}-\delta\mathbf{e}_i)<z(\bm{\beta})$ for all $\delta > 0$ for all $i\in\{1,\dots,m\}$ such that $\bm{\beta}-\delta\mathbf{e}_i\in\mathcal{B}$\color{black}.
For all $k \in \{1, ..., n\}$, $\bar{\mathbf{B}}_k\subset \mathbb{Z}^m_+$ is the set of level-set-minimal vectors with respect to $z_k$. Further, $\bar{\mathbf{B}} = \bar{\mathbf{B}}_n$ is the set of level-set-minimal vectors with respect to $z$.\color{black}
\end{defn}
We define $\bar{\mathcal{B}}$, the component-wise integral subset of the domain of $z$, as $\bar{\mathcal{B}} \coloneqq \cB \cap \mathbb{Z}^{m}_{+}$.
\begin{remark}
Note that if $\bm{\beta}\in\bar{\mathcal{B}}$, \thref{minimal} can be simplified as $\bm{\beta}$ is level-set-minimal if $z(\bbeta-\mathbf{e}_i) < z(\bbeta)$ for all $i\in \{1, ..., m\}$ such that $\bbeta-\mathbf{e}_i \in \bar{\mathcal{B}}$.
\end{remark}
\color{black}
Determining whether a vector is level-set-minimal is NP-complete (\cite{Trapp2014}). Thus, we provide a variety of necessary and sufficient conditions to verify whether a right-hand side is level-set-minimal.

\thref{equal} states that the set of level-set-minimal vectors is a subset of the image of $\mathbf{A}$ over $\mathbb{Z}^{n}_{+}$, which implies that all optimal solutions of a level-set-minimal vector are tight at all constraints. Note that for any $\bm{\beta}_1, \bm{\beta}_2\in \mathbb{R}^m$, we say that $\bm{\beta}_1\lneq \bm{\beta}_2$ if $\bm{\beta}_1\leq\bm{\beta}_2$, and $\bm{\beta}_1\neq\bm{\beta}_2$.

\begin{myl}\thlabel{equal}
Given $k \in \{1, ..., n\}$, for all $\bm{\beta}\in \bar{\mathbf{B}}_k$ and all $\mathbf{x}^*\in\opt_k(\bm{\beta})$, $\sum_{j=1}^k\mathbf{a}_j x^*_j=\bm{\beta}$.
\end{myl}

\thref{LSMiter} provides a sufficient condition under which a level-set-minimal vector with $k-1$ variables maintains level-set-minimality with $k$ variables.
\begin{mypro}\thlabel{LSMiter}
Let $\bm{\beta}\in \bar{\mathbf{B}}_{k-1}$. If for all $\bar{\bm{\beta}}\lneq \bm{\beta}$ and all $\mathbf{x}^* \in \opt_k(\bar{\bm{\beta}})$, $x^*_k= 0$, then $\bm{\beta}\in \bar{\mathbf{B}}_{k}$.
\end{mypro}
Note that there are some conditions under which $\bbeta \in \bar{\mathbf{B}}_{k-1}$ and $\bbeta \in \bar{\mathbf{B}}_k$ but there exists $\bar{\bbeta} \lneq \bbeta$ with $\mathbf{x}^* \in \opt_k(\bar{\bm{\beta}})$, $x^*_k > 0$. For example, if $\mathbf{a}_{k-1} = \mathbf{a}_k$ and $c_{k-1} = c_k$, then $\bar{\mathbf{B}}_{k-1} = \bar{\mathbf{B}}_k$, but any $\bar{\bbeta}$ with $\mathbf{x}^* \in \opt_k(\bar{\bm{\beta}})$, $x^*_k = 0$, $x^*_{k-1} > 0$ will also have $ \mathbf{x}^*- \mathbf{e}_{k-1}+\mathbf{e}_k$ as an optimal solution.\par
Parametrized IP \ref{EXIP} is used in examples throughout. Note that the constraint right-hand side upper bound, $\mathbf{b}$, is specified for each example, or $\cB$ is given to be unbounded above.
\begin{align*}\label{EXIP}
\max\limits_{\bx} \ &2x_1+3x_2+4x_3+3x_4+3x_5+6x_6\\
\emph{s.t.}\ &\begin{pmatrix}1 & 2 & 1 & 1 & 1 & 2\\ 1 & 1 & 2 & 1 & 3 & 2\end{pmatrix}\bx \leq \bbeta,\ \bx \in \mathbb{Z}^{6}_{+}.\tag{EXIP$(\bbeta)$}
\end{align*}
\begin{example}\thlabel{exExampleIP}
Consider \ref{EXIP} with $\mathbf{b} = (3,4)^{\top}$, and $\bbeta = (3,3)^{\top} \in \bar{\mathbf{B}}_4$. For all $\bar{\bbeta} \lneq \bbeta$ and for all $\mathbf{x}^* \in \opt_5(\bar{\bbeta})$, $x^*_5 = 0$. As such, $\bbeta \in \bar{\mathbf{B}}_5$.\hfill\qedsymbol\end{example}
\thref{LSMlinind} uses a linear independence relationship among a subset of primal variables to provide a necessary condition for level-set-minimal vectors.
\begin{mypro}\thlabel{LSMlinind}
If $\bm{\beta} \notin\bar{\mathbf{B}}_{k-1}$, $\mathbf{a}_1,\dots,\mathbf{a}_k$ are linearly independent, and $\bbeta$ and $\mathbf{a}_k$ are linearly independent, then for all $t\in \mathbb{Z}_+$ such that $\bm{\beta}+t\mathbf{a}_k\in \mathcal{B}$, $\bm{\beta}+t\mathbf{a}_k \notin \bar{\mathbf{B}}_{k}$.
\end{mypro}

\thref{LSMiter,LSMlinind} provide conditions for which a right-hand side is level-set-minimal as primal variables are added. In contrast, \thref{schaefer} gives a sufficient condition such that a right-hand side is level-set-minimal as primal variables are removed.

\begin{mypro}\thlabel{schaefer}
If $\bm{\beta}\in \bar{\mathbf{B}}_{k}$ and $\mathbf{x}^*\in \opt_k(\bm{\beta})$, then $\bm{\beta}-\mathbf{a}_kx^*_k\in \bar{\mathbf{B}}_{k-1}.$ 
\end{mypro}
\begin{proof}
Suppose $\bbeta - \ba_{k}x^{*}_{k} \not\in \barB_{k-1}$, then there exists $\bpi \in \barB_{k-1}$ such that $\bpi \lneq \bbeta - \ba_{k}x^{*}_{k}$ and $z_{k-1}(\bpi) = z_{k-1}(\bbeta - \ba_{k}x^{*}_{k})$. Let  $\bar{\bx} \in \opt_{k-1}(\bpi)$. Because $\bpi \lneq \bbeta - \ba_{k}x^{*}_{k},$ and $\sum\limits_{j = 1}^{k-1}c_{j}\bar{x}_{j} = z_{k-1}(\bpi) = z_{k-1}(\bbeta - \ba_{k}x^{*}_{k}),$ we have $\bar{\bx} \in \opt_{k-1}(\bbeta - \ba_{k}x^{*}_{k})$. Moreover, ($x^{*}_{1},\dots,x^{*}_{k-1})^{\top}$ is a feasible solution to IP$_{k-1}(\bbeta - \ba_{k}x^{*}_{k})$, which implies that $\slj^{k-1}c_{j}\bar{x}_{j} \geq \slj^{k-1}c_{j}x^{*}_{j}$. 

Let $\hat{\bx} =  (\bar{x}_{1},\dots,\bar{x}_{k-1}, x^{*}_{k})^{\top}$. Then $\slj^{k}c_{j}\hat{x}_{j} = \slj^{k-1}c_{j}\bar{x}_{j} + c_{k}x^{*}_{k} \geq \slj^{k}c_{j}x^{*}_{j}$, and because $\bx^{*} \in \opt_k(\bbeta), \hat{\bx} \in \opt_k(\bbeta)$. However, $\slj^{k}\ba_{j}\hat{x}_{j} = \slj^{k-1}\ba_{j}\bar{x}_{j} + \ba_{k}x^{*}_{k} = \bpi + \ba_{k}x^{*}_{k} \lneq \bbeta$, which implies that $\bbeta \not\in \barB_{k}$, a contradiction. 
\end{proof}
\begin{continueexample}{exExampleIP}
In (EXIP), if $\mathbf{b} = (3,4)^{\top}$, then $\bbeta = (3,3)^{\top} \in \bar{\mathbf{B}}_3$, and $(0,1,1)^{\top} \in \opt_3(\bbeta)$. Hence, $(3,3)^{\top}-\mathbf{a}_3 = (2,1)^{\top} \in \bar{\mathbf{B}}_2$.\hfill\qedsymbol\end{continueexample}

\begin{mypro}\thlabel{down}
If $\mathbf{x}^* \in \opt_k(\sum_{j=1}^k\mathbf{a}_jx^*_j)$ and $\sum_{j=1}^k\mathbf{a}_j x_j^*\in\bar{\mathbf{B}}_k$, then for all $\mathbf{x}\in\mathbb{Z}_+^k$ such that $\mathbf{x}\lneq \mathbf{x}^*$, \color{black} we have  $\sum_{j=1}^k\mathbf{a}_j{x}_j\in\bar{\mathbf{B}}_k$.

\end{mypro}
\begin{proof}Suppose there exists $\bx^{1} \lneq {\bx^*}$ such that $\slj^{k}\ba_{j}x^{1}_{j} \not\in \barB_{k}.$ Then there exists $\bx^{2} \in \opt_k(\slj^{k}\ba_{j}x^{1}_{j})$ such that $\slj^{k}\ba_{j}x^{2}_{j} \lneq \slj^{k}\ba_{j}x^{1}_{j}$. Let $\bx^{3} = \bx^{2} + ({\bx^*} - \bx^{1})$. Then, $\slj^{k} \ba_{j}x^{3}_{j} = \slj^{k}\ba_{j}x^{2}_{j} + \slj^{k}\ba_{j}{x}^*_{j} - \slj^{k}\ba_{j}x^{1}_{j} \lneq \slj^{k} \ba_{j}{x}_{j}^*.$ Because $\bx^{2} \in \opt_k(\slj^{k} \ba_{k}x^{1}_{j})$, $\slj^{k} c_{j}x^{3}_{j} = \slj^{k}c_{j}(x^{2}_{j} + {x}^*_{k} - x^{1}_{j}) \geq \slj^{k} c_{j}{x}^*_{j}$, and because ${\bx^*} \in \opt_k(\slj^{k}\ba_{j}{x}^*_{j}), \bx^{3} \in \opt_k(\slj\ba_{j}{x}^*_{j})$. However, $\slj^{k} \ba_{j}x^{3}_{j} \lneq \slj^{k}\ba_{j}{x}^*_{j}$, which contradicts \thref{equal}.
\end{proof}

\begin{mypro}\thlabel{anotinB}
Let $\bm{\beta}\in\bar{\mathbf{B}}_k$. If $\mathbf{a}_j\notin \bar{\mathbf{B}}_k$, then for all $ \mathbf{x}^* \in\opt_k(\bm{\beta})$, $x^*_j=0$.
\end{mypro}
\begin{cor}\thlabel{relationship}
$\bar{\mathbf{B}}_k\subseteq \left\{\bm{\beta}\in\bar{\mathcal{B}}\,\vert\, \bm{\beta}=\hat{\bm{\beta}}+t\mathbf{a}_k,\hat{\bm{\beta}}\in\bar{\mathbf{B}}_{k-1},t\in\mathbb{Z}_+ \right\}$.
\end{cor}
Propositions 7 and 8 state sufficient and necessary conditions, respectively, to ensure that subsets of right-hand sides are all level-set-minimal. \thref{down} requires a known optimal solution and generalizes a result of \cite{Trapp2014}, who prove the case of $k=n$. In contrast, \thref{anotinB} relies on the value function.

\begin{mypro}\thlabel{downt}
If there exists $t\in \mathbb{Z}_+$ such that $\bm{\beta}+t\mathbf{a}_k\in  \bar{\mathbf{B}}_{k}$ and $z_k(\bm{\beta}+t\mathbf{a}_k)=z_k(\bm{\beta})+tc_k$, then for all $s\leq t,s\in\mathbb{Z}_+$, $\bm{\beta}+s\mathbf{a}_k\in  \bar{\mathbf{B}}_{k}$.
\end{mypro}
\begin{proof}
Suppose there exists $s< t,s\in\mathbb{Z}_+$ such that $\bm{\beta}+s\mathbf{a}_k\notin  \bar{\mathbf{B}}_{k}$. Then there exists $\bm{\pi}\lneq\bm{\beta}+s\mathbf{a}_k$ such that $z_k(\bm{\pi})=z_k(\bm{\beta}+s\mathbf{a}_k)$. Then by superadditivity, 
\[z_k(\bm{\pi}+(t-s)\mathbf{a}_k)\geq z_k(\bm{\pi})+(t-s)c_k= z_k(\bm{\beta}+s\mathbf{a}_k)+(t-s)c_k\geq z_k(\bm{\beta})+tc_k=z_k(\bm{\beta}+t\mathbf{a}_k).\]
Because  $\bm{\pi}\lneq\bm{\beta}+s\mathbf{a}_k$ , $\bm{\pi}+(t-s)\mathbf{a}_k\lneq\bm{\beta}+t\mathbf{a}_k$, and $z_{k}(\bm{\pi} + (t - s)\ba_{k}) \geq z_{k}(\bbeta + t\ba_{k})$, then by monotonicity, $z_{k}(\bbeta + (t - s)\ba_{k}) \geq z_{k}(\bbeta + t\ba_{k})$, so that $\bbeta + t\ba_{k} \notin \bar{\mathbf{B}}_{k}$, a contradiction. 
\end{proof}

\begin{remark}
If $\bm{\beta}\in \bar{\mathbf{B}}_{k-1}$ and $\bm{\beta}\notin \bar{\mathbf{B}}_{k}$, then for all $t\in \mathbb{Z}_+$ such that $\bm{\beta}+t\mathbf{a}_k\in \bar{\mathcal{B}}$, either
\begin{enumerate}[label = (\arabic*)]
\item $\bm{\beta}+t\mathbf{a}_k\notin \bar{\mathbf{B}}_{k}$; or
\item $\bm{\beta}+t\mathbf{a}_k\in \bar{\mathbf{B}}_{k}$, $z_k(\bm{\beta}+t\mathbf{a}_k)>z_k(\bm{\beta})+tc_k$.  Further, there exist $\hat{\bm{\beta}}\in \bar{\mathbf{B}}_{k},\,\hat{t}\in \mathbb{Z}_+$ such that $\bm{\beta}+t\mathbf{a}_k=\hat{\bm{\beta}}+\hat{t}\mathbf{a}_k$ and $z_k(\bm{\beta}+t\mathbf{a}_k)=z_k(\hat{\bm{\beta}})+\hat{t}c_k$. 
\end{enumerate}
\end{remark}

Therefore, level-set-minimal vectors can be obtained by adding primal variables and the corresponding columns of $\mathbf{A}$ to the problem one at a time.

\subsection{Order of Primal Decision Variables}\label{varord}
We have shown a number of properties of value functions and level-set-minimal vectors when we introduce variables $x_{k}$ for a given ordering ($k = 1,2,\dots,n$) one at a time. However, if the ordering provided is arbitrary, then at each step $k$, it must first be determined if $\mathbf{a}_k$ should be added into the problem, or if it is strictly dominated by other variables and can be excluded from consideration. Here, we consider which variables are necessary to include to obtain an optimal solution; note that the order in which the decision variables are added can influence the construction of the level sets.  \par
\begin{myl}\cite{wolsey1981integer}\thlabel{>0}
If $z_n(\mathbf{a}_j)>c_j$, then for all $\bm{\beta} \in\mathcal{B}$ and all $\mathbf{x}^*\in \opt_n(\bm{\beta})$, $x^*_j=0$. 
\end{myl}

By \thref{>0} and \thref{anotinB}, only vectors $\mathbf{a}_j\in \bar{\mathbf{B}}$ such that $z_n(\mathbf{a}_j)=c_j$ are necessary to find an optimal solution for a given right-hand side.
\par
\begin{myl}\thlabel{LSMkton}
Suppose $\ba_{k} \in \bar{\bB}_{k}$ and $z_{k}(\ba_{k}) = c_{k}$, for some $k \in \{1,\dots,n\}$. Further suppose that for all $\hat{k} \in \mathbb{Z}_{+}$ such that $k < \hat{k} \leq n$, we have $\ba_{k} - \ba_{\hat{k}} \not\in \cB$. Then, $z_{n}(\ba_{k}) = c_{k}$ and $\ba_{k} \in \bar{\bB}_n$.
\end{myl}

By \thref{LSMkton}, the columns of $\mathbf{A}$ should be ordered such that $\mathbf{a}_1\lneq  \mathbf{a}_2\lneq \dots \mathbf{a}_n$. Some of these columns cannot be compared with each other, in which case it is sufficient to ensure that for all $k, \hat{k} \in \mathbb{N}$ where $k<\hat{k}\leq n$, $\mathbf{a}_k-\mathbf{a}_{\hat{k}}\notin \mathcal{B}$. By ordering the variables in this manner, it is sufficient to check if $z_k(\mathbf{a}_k)=c_k$ and $\mathbf{a}_k\in \bar{\mathbf{B}}_k$ to evaluate whether or not to add variable $k$ to the set during the $k^{th}$ step.
By \thref{nondecreasingbetweeniterations}, to check if $z_k(\mathbf{a}_k)=c_k$, we first check if $z_{k-1}(\mathbf{a}_k)\leq c_k$, then if $\mathbf{a}_k\in \bar{\mathbf{B}}_k$.\par
\thref{altvalue} demonstrates the possible relationships between $z_{k-1}(\mathbf{a}_k)$ and $z_k(\mathbf{a}_k)$, and the impact of this on the membership of $\mathbf{a}_k$ in $\bar{\mathbf{B}}_k$.
\begin{mypro}\thlabel{altvalue}
Exactly one of the following holds:
\begin{enumerate}[label = (\roman*)]
    \item $z_{k-1}(\mathbf{a}_k) < c_k$. Then $z_{k}(\mathbf{a}_k)= c_k$ and $\mathbf{a}_k\in \bar{\mathbf{B}}_k$.
    \item $z_{k-1}(\mathbf{a}_k) = c_k$. Then $z_{k}(\mathbf{a}_k)= c_k$.
    \item $z_{k-1}(\mathbf{a}_k) > c_k$. Then $z_{k}(\mathbf{a}_k)=z_{k-1}(\mathbf{a}_k)$.
\end{enumerate}
Thus, $z_k(\mathbf{a}_k) = \max\left\{z_{k-1}(\mathbf{a}_k), c_k\right\}$. In addition, if $z_{k-1}(\mathbf{a}_k) \geq c_k$, $\mathbf{a}_k\in \bar{\mathbf{B}}_k$ if and only if $\mathbf{a}_k\in \bar{\mathbf{B}}_{k-1}$.
\end{mypro}
\begin{continueexample}{exExampleIP}
In (EXIP), with $\mathbf{b} = (3,3)^{\top}$:
\begin{enumerate}[label=(\roman*)]
    \item $z_3(\mathbf{a}_3) = 4$ and $z_2(\mathbf{a}_3) = 2$, so $\mathbf{a}_3 \in \bar{\mathbf{B}}_3$,
    \item $z_5(\mathbf{a}_5) = 3$, $z_4(\mathbf{a}_5) = 3$, and $\mathbf{x}^* = \mathbf{e}_4 \in \opt_4(\mathbf{a}_5)$ has $\sum_{j=1}^{4}\mathbf{a}_jx^*_j<\mathbf{a}_5$, so $\mathbf{a}_5 \notin \bar{\mathbf{B}}_5$, and
    \item $z_6(\mathbf{a}_6) = 6$ and $\opt_5(\mathbf{a}_6) = \{2\mathbf{e}_4\}$, so for all $\mathbf{x} \in \opt_5(\mathbf{a}_6)$, $\sum_{j=1}^{5}\mathbf{a}_jx^*_j=\mathbf{a}_6$, so $\mathbf{a}_6 \in \bar{\mathbf{B}}_6$.\hfill\qedsymbol
\end{enumerate}\end{continueexample}
\thref{genstepup} generalizes \thref{altvalue} to relationships between $z_{k-1}(\bbeta)$ and $z_k(\bbeta)$ for arbitrary fixed $\bbeta$; in particular, \thref{genstepup} suggests a Gilmore-Gomory approach for solving IP$(\bbeta)$ or IP$_k(\bbeta)$. The classic Gilmore-Gomory recursion is (\cite{NemhauserWolsey1988}):
\begin{equation*}\label{GGR}
z(\bbeta) = \max\left\{z(\bbeta-\mathbf{a}_j)+c_j \,\vert\, j \in \{1, ..., n\}, \mathbf{a}_j\leq \bbeta\right\}.
\end{equation*}
\begin{mypro}\thlabel{genstepup}
Given $\bbeta \in \mathcal{B}$, $z_k(\bbeta) =
\max\limits_{\ell}\left\{z_{k-1}(\bbeta - \ell\ba_{k})+\ell c_{k} \,\vert\, \ell\ba_{k} \leq \bbeta, \ell \in \mathbb{Z}_{+}\right\}$.
\end{mypro}
\begin{proof}
Suppose $z_{k}(\bbeta) > \max\limits_{\ell}\left\{z_{k-1}(\bbeta - \ell\ba_{k})+\ell c_{k} \,\vert\, \ell\ba_{k} \leq \bbeta, \ell \in \mathbb{Z}_{+}\right\}$; equivalently, there exists $\mathbf{x}^*$ such that $\sum_{j=1}^k \mathbf{a}_jx_j^* \leq \bbeta$ and $\sum_{j=1}^k c_jx_j^* > \max\limits_{\ell\in \mathbb{Z}_+:\, \ell\mathbf{a}_k \leq \bbeta}z_{k-1}(\bbeta-\ell \mathbf{a}_k) + \ell c_k$. Then $\sum_{j=1}^{k-1}x_j^*c_j > \max\limits_{\ell\in \mathbb{Z}_+ :\, \ell\mathbf{a}_k \leq \bbeta}z_{k-1}(\bbeta-\ell \mathbf{a}_k) +  c_k(\ell-x_k^*) \geq z_{k-1}(\bbeta-x_k^* \mathbf{a}_k)$, a contradiction. On the other hand, suppose $z_{k}(\bbeta) < \max\limits_{\ell}\left\{z_{k-1}(\bbeta - \ell\ba_{k})+\ell c_{k} : \ell\ba_{k} \leq \bbeta,\ \ell \in \mathbb{Z}_{+}\right\}$; that is, there exists $\ell \in \mathbb{Z}_{+}$ such that $\ell\mathbf{a}_k \leq \bbeta$ and $z_{k-1}(\bbeta-\ell\mathbf{a}_k)+\ell c_k > z_k(\bbeta)$. Let $\mathbf{x}^* \in \opt_{k-1}(\bbeta-\ell\mathbf{a}_k)$, and define $\bx' \in \mathbb{Z}^{k}$ as $x'_{j} = x^{*}_{j}$ for $j \in \{1,\dots,k-1\}$ and $x'_{k} = \ell$. 
Then $\mathbf{x}'$ is feasible for IP$_k(\bbeta)$; thus, $z_{k}(\bbeta) \geq \sum\limits_{j = 1}^{k}c_{j}x'_{j} = z_{k-1}(\bbeta - \ell\ba_{k}) + \ell c_{k} > z_{k}(\bbeta)$, a contradiction. 
We conclude that $z_k(\bbeta) = \max\limits_{\ell\in\mathbb{Z}_+ :\, \ell\mathbf{a}_k \leq \bbeta}z_{k-1}(\bbeta-\ell \mathbf{a}_k) + \ell c_k$.
\end{proof}
The key differences between the two approaches are that \thref{genstepup} focuses on a single variable from the problem at each level of the recursion, while all variables remain present throughout the Gilmore-Gomory approach given in \cite{NemhauserWolsey1988}. In addition, the number of problems generated by \thref{genstepup} depends on the maximum number of $\mathbf{a}_k$ which can be removed from $\bbeta$, while the number of problems generated by the classic recursion is equal to the number of columns less than or equal to $\bbeta$ in a component-wise sense.

\section{Maximal Connected \color{black}Level Sets of the IP Value Function}
A level set of the IP value function may consist of multiple subsets of the hyperrectangle $\cB$ that are not connected (see \thref{Tstep} and Figures \ref{fig:MC-level-A} and \ref{fig:MC-level-B}). Hence, the structure of optimal solutions may vary greatly within the same level set. Having access to connected \color{black}subsets of level sets for the recourse value function in these problems could allow for optimization with respect to a connected set \color{black} of right-hand sides in each subset as a subproblem with fixed second-stage value and known allowable variability with respect to those bounds. We explore connected \color{black}subsets of level sets, and in particular, we define and discuss properties of maximal connected \color{black}subsets of level sets.
Maximal connected \color{black}level lattices (MC-level lattices) are MC-level sets intersected with the lattice $\mathbb{Z}^{m}_+$\color{black}. Note that because we are interested in the connections between sets of right-hand sides, we assume throughout this section that $\cB = \mathbb{R}^m_+$ (or, equivalently, that each component of $\mathbf{b}$ is positive infinity). The removal of this assumption affects only \thref{MC_Adjacent}, which does not necessarily hold for MC-level sets that intersect an upper boundary plane of $\cB$.\color{black}
\begin{defn}\thlabel{ivcurve}
Given $\bbeta_1$, $\bbeta_2 \in \cB$, a continuous function $d:[0,1]\to\cB$ is a \textbf{continuous isovalue curve} from $\bbeta_1$ to $\bbeta_2$ if $d(0) = \bbeta_1$, $d(1) = \bbeta_2$, and $z(d(t)) = z(\bbeta_1)$ for all $t\in [0,1]$.
\end{defn}
Given disjoint connected subsets $M_1, M_2 \subset [0,1]$, we say that $M_1$ precedes $M_2$ in the ordering of $[0,1]$ if for any $t_1 \in M_1$, $t_2 \in M_2$, $t_1 < t_2$.\color{black}
\begin{defn}\thlabel{Tstep} For all ${\bm{\beta}} \in \cB$, define the \textbf{MC-level set} $T({{\bm{\beta}}}) \coloneqq \left\{\bar{\bm{\beta}}\in\cB\,\vert\right.$ there exists a continuous isovalue curve $d$ from ${\bbeta}$ to $\left.\bar{\bbeta}\right\}$. For all ${\bm{\beta} \in }$ \wenxin{$\bar{\mathcal{B}}$}, \color{black} define the \textbf{MC-level lattice} $C({\bm{\beta}}) \coloneqq T({\bm{\beta}})\cap \mathbb{Z}^m_+ $.
\end{defn}
\begin{continueexample}{exExampleIP}
In (EXIP), with $\cB$ unbounded above, $T((1,1)^{\top})$ is the unbounded set $\left\{\bbeta\,\vert\, \bbeta_1 \geq 1, \ 1\leq \bbeta_2 < 2\right\}$, and $C((1,1)^{\top}) =\left\{(v,1)\,\vert\,v \in \mathbb{N}\right\}$\color{black}.\hfill\qedsymbol\end{continueexample}
\begin{myl}\thlabel{easyiso}
Given $\bbeta_1\in \cB$, $\bbeta_2 \in T(\bbeta_1)$, there exists an isovalue curve $d': [0,1]\to\cB$ from $\bbeta_1$ to $\bbeta_2$ such that $\lfloor d'(\zeta)\rfloor = \lfloor d'(\theta)\rfloor$ implies $\lfloor d'(\zeta)\rfloor = \lfloor d'(\eta)\rfloor$ for all $0\leq \zeta<\eta<\theta\leq 1$  - that is, $\lfloor d'\rfloor$ takes on any given value for at most a single connected subset of $[0,1]$.
\end{myl}\color{black}

\begin{defn}
Given $\bbeta_1,\bbeta_2 \in \mathbb{Z}^m_+$, $\bbeta_1$ and $\bbeta_2$ are \textbf{adjacent} if there exists $j \in \{1, ..., m\}$ such that $\bbeta_1-\bbeta_2 = \pm\mathbf{e}_j$.
\end{defn}

\thref{Csteppro} shows that there exists a sequence of adjacent points in an MC-level lattice which connects any two of its members.
\begin{mypro}\thlabel{Csteppro}
Given ${\bm{\beta}}\in \bar{\mathcal{B}}$ and $\bm{\beta}_1$, $\bm{\beta}_2 \in C({\bm{\beta}})$, there exists a finite sequence of points, $V = \left\{v_1 = \bm{\beta}_1, ..., v_r = \bm{\beta}_2\right\}$, such that for all $i\in \{1, ..., r\}$, $v_i\in C({\bm{\beta}})$ and for all $i \in \{1, ..., r-1\}$, $v_i$ and $v_{i+1}$ are adjacent.
\end{mypro}
\begin{proof}
By \thref{Tstep} and \thref{easyiso}, there must exist a continuous isovalue curve $d':[0,1]\to\cB$ from $\bbeta_1$ to $\bbeta_2$ for which $\lfloor d'\rfloor$ takes on any given value for at most a single connected subset of $[0,1]$. Let $\Gamma = \left\{\bm{\gamma}^1 = \bbeta_1, ..., \bm{\gamma}^s = \bbeta_2\right\}$ be the ordered sequence of values taken on by $\lfloor d'\rfloor$, so that $\lfloor d(t_i)\rfloor = \bm{\gamma}^i$, $\lfloor d(t_j)\rfloor = \bm{\gamma}^j$, $t_i, t_j \in [0,1]$, and $i < j$ collectively imply $t_i < t_j$. Let $\mathbf{M} = \{M_1, ..., M_s\}$ be the ordered sequence of connected subsets of $[0,1]$ corresponding to the members of $\Gamma$ - that is, for any $i \in \{1,\dots,s\}$, $t \in M_i$ is equivalent to $\lfloor d'(t)\rfloor = \bm{\gamma}^i$. Denote the closure of $M_i$ as $\overline{M}_i$, and define {\tajayi{ the singleton $\tau_i$ such that $\{\tau_{i}\} = \overline{M}_i\cap\overline{M}_{i+1}$; because the members of $\mathbf{M}$ are disjoint and ordered, and their union is $[0,1]$, distinct $\tau_i$ will exist for each $i\in\{1, ..., s-1\}$.}} Note that either $\tau_i \in M_i$ or $\tau_i \in M_{i+1}$, and that since $d'$ is a continuous function over a bounded domain, each component of $d'$ is also bounded, so that $\Gamma$ is finite. Further, since $d'$ is continuous, $\Vert \bm{\gamma}^{i+1}-\bm{\gamma}^i\Vert_{\infty} = 1$ for all $i \in \{1, ..., s-1\}$.\par
Suppose there exists $i \in \{1, ..., s-1\}$ for which for some $j \in \{1, ..., m\}$, $\bm{\gamma}^{i+1}_j > \bm{\gamma}^i_j$ and for some $k \in \{1, ..., m\}$, $\bm{\gamma}^i_k > \bm{\gamma}^{i+1}_k$. There are two possibilities:

\vspace{12pt}
\setlength{\leftskip}{1cm}
Case 1: Suppose that both $M_i$ and $M_{i+1}$ include more than a single point. Since $d'$ is continuous, there must exist $\epsilon > 0$ such that for all $t \in (\tau_i-\epsilon, \tau_i)$, $\lfloor d'(t)\rfloor = \bm{\gamma}^i$ and for all $t \in (\tau_i, \tau_i+\epsilon)$, $\lfloor d'(t)\rfloor = \bm{\gamma}^{i+1}$. For any $t$ such that $\lfloor d'(t)\rfloor = \bm{\gamma}^i$, we must have $(\bm{\gamma}^{i+1}-d'(t))_j > 0$, so by the continuity of $d'$, we must have $d'(\tau_i)_j = \bm{\gamma}^{i+1}_j$. Similarly, for any $t$ such that $\lfloor d'(t)\rfloor = \bm{\gamma}^{i+1}$, we must have $(\bm{\gamma}^i-d'(t))_k > 0$, so that by the continuity of $d'$, we must have $d'(\tau_i)_k = \bm{\gamma}^i_k$. However, this implies that $\tau_i \notin M_i\cup M_{i+1}$, a contradiction. 

Case 2: Suppose on the other hand that $M_i$ is a singleton; then $M_{i} = \{\tau_i\}$. {\tajayi{Observe that, because $M_{i} \cap M_{i+1} = \emptyset$, if $M_{i+1}$ is a single point, $\overline{M_{i+1}} \cap \overline{M_{i}} = \emptyset$. This contradicts the continuity of $d'$; thus, $M_{i+1}$ contains more than a single point. Fix $\epsilon > 0$ such that $(\tau_{i}, \tau_{i} + \epsilon) \subseteq M_{i+1}$. For any $\delta \in (0, \epsilon)$, $d'(\tau_{i} + \delta)_{j} \geq \lfloor d'(\tau_{i} + \delta)_{j}\rfloor = \bm{\gamma}^{i+1}_{j} = \bm{\gamma}^{i}_{j} + 1$. By continuity, $d'(\tau_{i})_{j} = \lim\limits_{t \to \tau_{i}^{+}} d'(t)_{j} \geq \gamma^{i}_{j} + 1 = \lfloor d'(\tau_{i})_{j}\rfloor + 1 > d'(\tau_{i})_{j}$, a contradiction. An analogous argument leads to a similar contradiction when $M_{i+1}$ is a singleton. Thus, we conclude that for all $i \in \{1,\dots,s\},$ either $\gamma^{i} \leq \gamma^{i+1}$ or $\gamma^{i} \geq \gamma^{i+1}$.}}\par
\vspace{12pt}
\setlength{\leftskip}{0cm}
Define $\hat{\bm{\gamma}}^j$ such that $\hat{\bm{\gamma}}^j_k = \mathbbm{1}_{k\leq j}\bm{\gamma}^i_k + \mathbbm{1}_{k> j}\bm{\gamma}^{i+1}_k$ for all $k \in \{1, ..., m\}$, and for all $j \in \{0, ..., m\}$. Then for all $j \in \{0, ..., m\}$, $\bm{\gamma}^i\leq \hat{\bm{\gamma}}^j\leq \bm{\gamma}^{i+1}$ or $\bm{\gamma}^i\geq\hat{\bm{\gamma}}^j\geq\bm{\gamma}^{i+1}$, and for all $j \in \{0, ..., m-1\}$, either $ \hat{\bm{\gamma}}^j-\hat{\bm{\gamma}}^{j+1} =\pm \mathbf{e}_j$ or $\hat{\bm{\gamma}}^j-\hat{\bm{\gamma}}^{j+1} = \mathbf{0}$. By the monotonicity of the value function, $\hat{\bm{\gamma}}^j \in C(\bbeta)$ for all $j \in \{0, ..., m\}$. Inserting $\left\{\hat{\bm{\gamma}}^1, ..., \hat{\bm{\gamma}}^{m-1}\right\}$ in $\Gamma$ between $\bm{\gamma}^i$ and $\bm{\gamma}^{i+1}$ for each $i \in \{1, ..., s-1\}$, then removing sequentially adjacent duplicate values, will yield a sequence $V$ satisfying the conditions of the proposition statement.
\end{proof}\color{black}
\thref{finseq} shows an analogous result for any pair of points in an MC-level set.
\begin{cor}\thlabel{finseq}
Given ${\bm{\beta}}\in\cB$, $\bm{\beta}_a$ and $\bm{\beta}_b\in T({\bm{\beta}})$, there exists a finite sequence of points,\\$\left\{\bm{\gamma}^0 = \bm{\beta}_a, ..., \bm{\gamma}^r = \bm{\beta}_b\right\}$, such that $\bm{\gamma}^i\in T({\bm{\beta}})$ and $\vert \bm{\gamma}^{i+1}-\bm{\gamma}^i\vert = \lambda_i\mathbf{e}_j$ for all $i$, where $\lambda_i \in (0,1]$ for all $i$ and $\mathbf{e}_j$ is a unit vector for $j \in \{1, ..., m\}$.
\end{cor}
Note that other than the steps necessary to move from $\bbeta_a$ and $\bbeta_b$ to the nearest isovalue lattice points, all the intermediate points in the above sequence may be made to be lattice points, as can be seen in Figure \ref{fig:MC-level-C}.
\begin{figure}[H]
    \centering
    \includegraphics[width=0.5\textwidth]{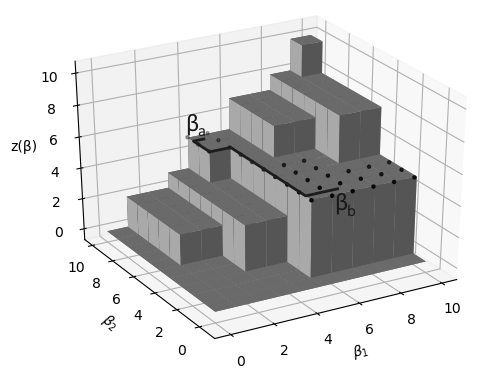} 
    \vspace{-12pt}
    \caption{An isovalue path from $\bbeta_a$ to $\bbeta_b$}
    \label{fig:MC-level-C}
\end{figure}\par
\begin{defn}\thlabel{adjhc}
An $m$-dimensional unit hypercube $H$ is \textbf{anchored} if all of its vertices are in $\mathbb{Z}^m$. In particular, $H(\bbeta)$, $\bbeta \in \mathbb{Z}^m$, is the anchored unit hypercube with $\bbeta$ as its vertex with all components minimal. 
\end{defn}
\begin{defn}
The \textbf{open-ceiling} unit hypercube anchored at $\bbeta \in \bar{\cB}$ is defined as $G(\bbeta) \coloneqq$ $\{\bar{\bbeta} \in$ $\cB\, \vert\,\lfloor \bar{\bbeta}\rfloor = \bbeta\}$.
\end{defn}
\begin{remark}
Given an open-ceiling unit hypercube $G(\bbeta)$, the closure of $G(\bbeta)$ is $H(\bbeta)$.
\end{remark}
\begin{defn}
Two anchored $m-dimensional$ unit hypercubes $H_1$ and $H_2$ are \textbf{adjacent} if $H_1\cap H_2$ is an anchored unit hypercube of dimension $m-1$.
\end{defn}
\thref{MC_Adjacent} shows that the closure of any MC-level set can be written as the union of a set of $m$-dimensional unit hypercubes anchored at integer points.
\begin{mypro}\thlabel{MC_Adjacent}
Given $\bbeta \in \cB\subset \mathbb{R}^m$, there exists a unique \color{black}set $P$ of anchored \color{black}unit hypercubes of dimension $m$ with integer-valued vertex coordinates such that $\bigcup\limits_{H \in P}H = \overline{T(\bbeta)}$, the closure of $T(\bbeta)$. Further, for all pairs $H_1$, $H_2\,\in P$, there exists a finite sequence $U = \{U_1, ..., U_t\}$ of hypercubes in $P$ such that $U_1 = H_1$, $U_t = H_2$, and $U_i$ and $U_{i+1}$ are adjacent for all $i \in \{1, ..., t-1\}$.\color{black}
\end{mypro}
\begin{proof}
Fix $\bbeta' \in T(\bbeta)$. Then the open-ceiling unit hypercube $G(\lfloor \bbeta'\rfloor)$ is a subset of $T(\bbeta)$, and $\bbeta' \in G(\lfloor \bbeta'\rfloor)$. As such, there exists a unique countable set $Q$ of $m$-dimensional open-ceiling anchored unit hypercubes, $Q\coloneqq \{Q_{r}\}_{r \in R}$, where $R$ is the index set for $Q$, such that $T(\bbeta) = \bigcup\limits_{r\in R}Q_r$ - that is, for all $\bar{\bbeta}\in T(\bbeta)$, there exists $r \in R$ such that $Q_r = G(\lfloor \bar{\bbeta}\rfloor)$. Let $P = \{\overline{Q_{r}}\}_{r \in R}$, and observe that $\overline{T(\bbeta)} = \bigcup\limits_{r \in R} \overline{Q_{r}}$ - that is, $P$ is a unique set of $m$-dimensional anchored unit hypercubes whose union is the closure of $T(\bbeta)$.\par
Next, fix $H_1$, $H_2$ in $P$. Let $\bbeta_1$ and $\bbeta_2$ anchor $H_1$ and $H_2$, respectively. Let $V = \{v_1 = \bbeta_1, ..., v_s = \bbeta_2\}$ be the sequence of adjacent points whose existence is guaranteed by \thref{Csteppro} given $\bbeta_1$ and $\bbeta_2$ as input points to connect, and define a finite sequence $U = \{H(v_1), ..., H(v_s)\}$, where $H(v_1) = H_1$ and $H(v_s) = H_2$. Since each of these points is in $T(\bbeta)$, the $m$-dimensional hypercube anchored at each point is a subset of $\overline{T(\bbeta)}$ and thus is an element of $P$. Further, since $v_i-v_{i+1}$ is a unit vector for all $i$, $H(v_i)\cap H(v_{i+1})$ is a unit anchored hypercube of dimension $m-1$, so $H(v_i)$ and $H(v_{i+1})$ are adjacent for all $i \in \{1, ..., s-1\}$.
\end{proof}
\color{black}\thref{MC_Adjacent} shows that the closure of an MC-level set can be constructed using anchored \color{black}unit hypercubes of dimension $m$\color{black}. In addition, given any two anchored unit hypercubes in that closure, we can find a finite sequence of adjacent anchored unit hypercubes connecting the two. In Figure \ref{fig:MC-level-A}, there is a sequence of adjacent dark squares between any two dark squares, since $T(\bbeta_1) = T(\bbeta_2)$. On the other hand, in Figure \ref{fig:MC-level-B}, although the closures of $T(\bbeta_4)$ and $T(\bbeta_5)$ intersect, they do not actually share a point - and therefore there is no sequence of isovalue adjacent squares from $T(\bbeta_4)$ to $T(\bbeta_5)$.
\begin{figure}[H]{
    \begin{minipage}[b]{0.43\textwidth}
    \centering
    \includegraphics[width=1\textwidth]{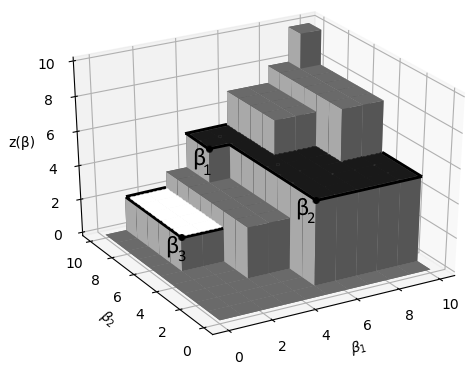} 
    \caption{$T({\bm{\beta}}_1)= T({\bm{\beta}}_2)$}
    \label{fig:MC-level-A}
\end{minipage}
\hspace{0.5cm}
    \begin{minipage}[b]{0.43\textwidth}
    \centering
    \includegraphics[width=1\textwidth]{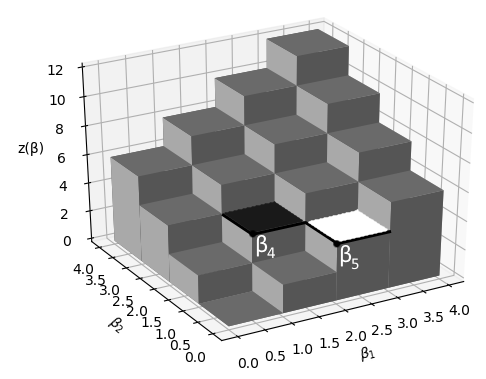}
    \caption{$T({\bm{\beta}}_4)\cap T({\bm{\beta}}_5)= \emptyset$}
    \label{fig:MC-level-B}
    \end{minipage}
}\end{figure}
\par
We next demonstrate some membership properties of MC-level sets based on level-set minimal vectors.
\begin{cor}\thlabel{segment}
Given $\bar{\bm{\beta}}\in \cB$, for all $\hat{\bbeta}\in \cB$ such that $z(\bar{\bm{\beta}}) = z(\hat{\bbeta})$ and $\hat{\bbeta}\geq\bar{\bm{\beta}}$, we have $\hat{\bbeta}\in T(\bar{\bm{\beta}})$.
\end{cor}
\begin{proof}
By the monotonicity of $z$, $z(\hat{\bbeta}) = z(\bar{\bbeta}) \leq z(\bbeta') \leq z(\hat{\bbeta})$, for all $\bbeta' \in \left\{\bbeta \ \vert \ \bbeta = \lambda\bar{\bbeta} + (1 - \lambda)\hat{\bbeta}, \lambda \in (0,1)\right\}$ because $\bbeta' \in (\bar{\bbeta}, \hat{\bbeta})$. As such, $d(t) \coloneqq t\bar{\bbeta}+(1-t)\hat{\bbeta}$ is a continuous isovalue curve from $\bar{\bbeta}$ to $\hat{\bbeta}$, so $\hat{\bbeta}\in T(\bar{\bbeta})$.
\end{proof}
\begin{mypro}\thlabel{LSMcor}
For any $\bar{\bm{\beta}}\in\cB$, $\bm{\beta}\in T(\bar{\bm{\beta}})$ if and only if $z(\bm{\beta}) = z(\bar{\bm{\beta}})$ and there exists $\hat{\bm{\beta}} \in $ \wenxin{$\bar{\mathbf{B}}$} \color{black} $\cap T(\bar{\bm{\beta}})$ such that $\hat{\bm{\beta}}\leq \bm{\beta}$.
\end{mypro}
\begin{cor}\thlabel{LSMseq}
Let \wenxin{$\bar{\bbeta} \in \bar{\mathbf{B}}$} \color{black} such that there exists $\hat{\bbeta} \in \bar{\mathbf{B}}\cap T(\bar{\bm{\beta}})$, $\hat{\bbeta}\neq \bar{\bbeta}$. Then there exists a sequence $V = \{v_1, ..., v_s\}$ of distinct vectors in $\bar{\mathbf{B}}\cap T(\bar{\bm{\beta}})$ such that $\bar{\bbeta} = v_1$, $\hat{\bbeta} = v_s$, and for all $i \in \{1, ..., s-1\}$, there exists $\bbeta_i \in T(\bar{\bbeta})$ such that $v_i\lneq \bbeta_i$ and $v_{i+1}\lneq \bbeta_i$.
\end{cor}
\color{black}
\thref{LTstepdownt} discusses the relationship between a sequence of right-hand sides when ``stepping down" towards MC-level sets with smaller function value.
\begin{remark}\thlabel{optstat}
For any ${\bbeta} \in \cB$, let $\mathbf{x}^*\in \opt({\bbeta})$. Then for all $\bar{\bbeta} \in T({\bbeta})$ such that $\bar{\bbeta}\geq{\bbeta}$, $\mathbf{x}^*\in\opt(\bar{\bbeta})$.
\end{remark}
\begin{mypro}\thlabel{LTstepdownt}
For any $\bar{\bm{\beta}}\in$ \wenxin{$\bar{\mathbf{B}}$} \color{black} such that $z(\bar{\bm{\beta}}) > 0$, given ${\mathbf{x}}^* \in \opt(\bar{\bm{\beta}})$ and $\bbeta \in T(\bar{\bbeta})$ such that $\bar{\bbeta}\leq \bbeta$, then for all $j$ such that ${x}^*_j > 0$ and all $t \leq {x}^*_j$, $\bm{\beta}-t\mathbf{a}_j\in T(\bar{\bm{\beta}}-t\mathbf{a}_j)$.
\end{mypro}
\begin{proof}
Note that if $\bbeta = \bar{\bbeta}$ the result holds trivially. Suppose there exists $\hat{\bm{\beta}}\in T(\bar{{\bm{\beta}}})$, $\bar{\bm{\beta}}\lneq\hat{\bm{\beta}}$, such that there exists $t:\,1\leq t\leq {x}^*_j$ for which $\hat{\bm{\beta}}-t\mathbf{a}_j\notin T(\bar{\bm{\beta}}-t\mathbf{a}_j)$. Then there exists $\bm{\pi}\in\bar{\mathbf{B}}$ such that $\bar{\bm{\beta}}-t\mathbf{a}_j\lneq \bm{\pi}\leq \hat{\bm{\beta}}-t\mathbf{a}_j$ and \wenxin{$z(\bar{\bm{\beta}}-t\mathbf{a}_j)<z(\bm{\pi})=z(\hat{\bm{\beta}}-t\mathbf{a}_j)$}. \color{black}Let $\hat{\mathbf{x}}\in\opt(\bm{\pi})$. Then \wenxin{$z(\hat{\bm{\beta}}) \geq z(\bm{\pi}+t\mathbf{a}_j) \geq z(\bm{\pi})+tc_j$}. \color{black} However, by \thref{equal} and \thref{down},  $z(\bar{\bm{\beta}}-t\mathbf{a}_j) = z(\bar{\bm{\beta}})-tc_j = z(\hat{\bm{\beta}})-tc_j$, a contradiction since we have now claimed that $z(\hat{\bm{\beta}})-tc_j < z(\bm{\pi}) \leq z(\hat{\bm{\beta}})-tc_j$.
\end{proof}
\begin{continueexample}{exExampleIP}
    In (EXIP), with $\cB$ unbounded above, $(1,1)^{\top} \in \bar{\mathbf{B}}_6$ and $z((1,1)^{\top}) = 3$, with the unique optimal solution $\mathbf{e}_4$. Further, $(2,1)^{\top}, (3,1)^{\top}\in T((1,1)^{\top})$. So by \thref{LTstepdownt}, $(1,0)^{\top}, (2,0)^{\top} \in T((0,0)^{\top})$.\hfill\qedsymbol\end{continueexample}
\thref{bk=b} indicates that if, for any $k \in \{1, ..., n\}$, $\bar{\mathbf{B}}_k=\bar{\mathcal{B}}$, then for all $k' \geq k$, all MC-level sets will be unit hypercubes.
\begin{mypro}\thlabel{bk=b}
If there exists $k\in\{1,\dots,n\}$ such that $\bar{\mathbf{B}}_k=\bar{\mathcal{B}}$, then for all $\hat{k}\geq k,\hat{k}\in \{1,\dots,n\}$, $\bar{\mathbf{B}}_{\hat{k}}=\bar{\mathcal{B}}$.
\end{mypro}

\section*{Acknowledgments}
The authors would like to thank the referee and associate editor for their thorough and valuable input. The authors would also like to thank Eric Antley, David Mildebrath, Saumya Sinha, and Silviya Valeva of Rice University for their helpful comments. This research was supported in part by National Science Foundation, USA grants CMMI-1826323 and CMMI-1933373. 

\bibliographystyle{plainnat}
\bibliography{ref.bib}

\begin{thebibliography}{21}
\providecommand{\natexlab}[1]{#1}
\providecommand{\url}[1]{\texttt{#1}}
\expandafter\ifx\csname urlstyle\endcsname\relax
  \providecommand{\doi}[1]{doi: #1}\else
  \providecommand{\doi}{doi: \begingroup \urlstyle{rm}\Url}\fi

\bibitem[Ahmed et~al.(2004)Ahmed, Tawarmalani, and Sahinidis]{Ahmed2003}
S.~Ahmed, M.~Tawarmalani, and N.~V. Sahinidis.
\newblock A finite branch-and-bound algorithm for two-stage stochastic integer
  programs.
\newblock \emph{Mathematical Programming}, 100\penalty0 (2):\penalty0 355--377,
  2004.

\bibitem[Ajayi et~al.(2020)Ajayi, Thomas, and Schaefer]{Ajayi2020}
T.~Ajayi, C.~Thomas, and A.~J. Schaefer.
\newblock The gap function: Evaluating integer programming models over multiple
  right-hand sides.
\newblock \emph{Operations Research (To appear)}, 2020.

\bibitem[Basu et~al.(2021)Basu, Ryan, and Sankaranarayanan]{Basu2018}
A.~Basu, C.~T. Ryan, and S.~Sankaranarayanan.
\newblock Mixed-integer bilevel representability.
\newblock \emph{Mathematical Programming}, 185:\penalty0 163--197, 2021.

\bibitem[Blair(1995)]{Blair95}
C.~E. Blair.
\newblock A closed-form representation of mixed-integer program value
  functions.
\newblock \emph{Mathematical Programming}, 71\penalty0 (2):\penalty0 127--136,
  1995.

\bibitem[Blair and Jeroslow(1977)]{Blair77}
C.~E. Blair and R.~G. Jeroslow.
\newblock The value function of a mixed integer program: {I}.
\newblock \emph{Discrete Mathematics}, 19\penalty0 (2):\penalty0 121--138,
  1977.

\bibitem[Blair and Jeroslow(1982)]{Blair82}
C.~E. Blair and R.~G. Jeroslow.
\newblock The value function of an integer program.
\newblock \emph{Mathematical Programming}, 23\penalty0 (1):\penalty0 237--273,
  1982.

\bibitem[Gilmore and Gomory(1966)]{gg1966}
P.~C. Gilmore and R.~E. Gomory.
\newblock The theory and computation of knapsack functions.
\newblock \emph{Operations Research}, 14\penalty0 (6):\penalty0 1045--1074,
  1966.

\bibitem[K{\i}l{\i}n{\c{c}}-Karzan et~al.(2009)K{\i}l{\i}n{\c{c}}-Karzan,
  Toriello, Ahmed, Nemhauser, and Savelsbergh]{kk2009}
F.~K{\i}l{\i}n{\c{c}}-Karzan, A.~Toriello, S.~Ahmed, G.~Nemhauser, and
  M.~Savelsbergh.
\newblock Approximating the stability region for binary mixed-integer programs.
\newblock \emph{Operations Research Letters}, 37\penalty0 (4):\penalty0
  250--254, 2009.

\bibitem[Kong et~al.(2006)Kong, Schaefer, and Hunsaker]{Kong2006}
N.~Kong, A.J. Schaefer, and B.~Hunsaker.
\newblock Two-stage integer programs with stochastic right-hand sides: A
  superadditive dual approach.
\newblock \emph{Mathematical Programming}, 108\penalty0 (2):\penalty0 275--296,
  Sep 2006.

\bibitem[Llewellyn and Ryan(1993)]{Llewellyn1993}
D.~C. Llewellyn and J.~Ryan.
\newblock A primal dual integer programming algorithm.
\newblock \emph{Discrete Applied Mathematics}, 45\penalty0 (3):\penalty0 261 --
  275, 1993.

\bibitem[Lozano and Smith(2017)]{ls2017}
L.~Lozano and J.~C. Smith.
\newblock A value-function-based exact approach for the bilevel mixed-integer
  programming problem.
\newblock \emph{Operations Research}, 65\penalty0 (3):\penalty0 768--786, 2017.

\bibitem[Nemhauser and Wolsey(1988)]{NemhauserWolsey1988}
G.~L. Nemhauser and L.~A. Wolsey.
\newblock \emph{Integer and Combinatorial Optimization}.
\newblock Wiley-Interscience. John Wiley \& Sons, 1988.

\bibitem[{\"O}zalt{\i}n et~al.(2012){\"O}zalt{\i}n, Prokopyev, and
  Schaefer]{Ozaltin2012}
O.~Y. {\"O}zalt{\i}n, O.~A. Prokopyev, and A.~J. Schaefer.
\newblock Two-stage quadratic integer programs with stochastic right-hand
  sides.
\newblock \emph{Mathematical Programming}, 133\penalty0 (1):\penalty0 121--158,
  Jun 2012.

\bibitem[Ralphs and Hassanzadeh(2014)]{Ralphs2014}
T.~K. Ralphs and A.~Hassanzadeh.
\newblock On the value function of a mixed integer linear optimization problem
  and an algorithm for its construction.
\newblock \emph{COR$@$L Technical Report 14T–004}, 2014.

\bibitem[Schultz et~al.(1998)Schultz, Stougie, and Van Der~Vlerk]{ss1998}
R.~Schultz, L.~Stougie, and M.~H. Van Der~Vlerk.
\newblock Solving stochastic programs with integer recourse by enumeration: A
  framework using {G}r{\"o}bner basis.
\newblock \emph{Mathematical Programming}, 83\penalty0 (1-3):\penalty0
  229--252, 1998.

\bibitem[Tavasl\text{\i}o\u{g}lu et~al.(2019)Tavasl\text{\i}o\u{g}lu,
  Prokopyev, and Schaefer]{Onur2019}
O.~Tavasl\text{\i}o\u{g}lu, O.~A. Prokopyev, and A.~J. Schaefer.
\newblock Solving stochastic and bilevel mixed-integer programs via a
  generalized value function.
\newblock \emph{Operations Research}, 67\penalty0 (6):\penalty0 1659--1677,
  2019.

\bibitem[Trapp and Prokopyev(2015)]{Trapp2015}
A.~C. Trapp and O.~A. Prokopyev.
\newblock A note on constraint aggregation and value functions for two-stage
  stochastic integer programs.
\newblock \emph{Discrete Optimization}, 15:\penalty0 37--45, 2015.

\bibitem[Trapp et~al.(2013)Trapp, Prokopyev, and Schaefer]{Trapp2014}
A.~C. Trapp, O.~A. Prokopyev, and A.~J. Schaefer.
\newblock On a level-set characterization of the value function of an integer
  program and its application to stochastic programming.
\newblock \emph{Operations Research}, 61\penalty0 (2):\penalty0 498--511, 2013.

\bibitem[Wang and Xu(2017)]{wx2017}
L.~Wang and P.~Xu.
\newblock The watermelon algorithm for the bilevel integer linear programming
  problem.
\newblock \emph{SIAM Journal on Optimization}, 27\penalty0 (3):\penalty0
  1403--1430, 2017.

\bibitem[Williams(1996)]{Williams1996}
H.~P. Williams.
\newblock Constructing the value function for an integer linear programme over
  a cone.
\newblock \emph{Computational Optimization and Applications}, 6\penalty0
  (1):\penalty0 15--26, 1996.

\bibitem[Wolsey(1981)]{wolsey1981integer}
L.~A. Wolsey.
\newblock Integer programming duality: Price functions and sensitivity
  analysis.
\newblock \emph{Mathematical Programming}, 20\penalty0 (1):\penalty0 173--195,
  1981.

\end{thebibliography}

\newpage
\appendix
\section*{Appendix}
\subsection*{Omitted Proofs}
\begin{makeshiftResult}{Lemma}{\ref{equal}}
Given $k \in \{1, ..., n\}$, for all $\bm{\beta}\in \bar{\mathbf{B}}_k$ and all $\mathbf{x}^*\in\opt_k(\bm{\beta})$, $\sum_{j=1}^k\mathbf{a}_j{x}^*_j=\bm{\beta}$. 
\end{makeshiftResult}
\begin{proof}
Suppose $\sum_{j=1}^k\mathbf{a}_j{x}^*_j\lneq\bm{\beta}$. Then $z_k(\sum_{j=1}^k\mathbf{a}_j{x}^*_j)=z_k(\bm{\beta})$, which contradicts $\bm{\beta}\in \bar{\mathbf{B}}_k$.
\end{proof}

\begin{makeshiftResult}{Proposition}{\ref{LSMiter}}
Let $\bm{\beta}\in \bar{\mathbf{B}}_{k-1}$. If for all $\bar{\bm{\beta}}\lneq \bm{\beta}$ and all $\mathbf{x}^*\in \opt_k(\bar{\bm{\beta}})$, ${x}^*_k= 0$, then $\bm{\beta}\in \bar{\mathbf{B}}_{k}$.
\end{makeshiftResult}
\begin{proof}
Let $\bm{\beta}\in S_k(\alpha)$. Suppose $\bm{\beta}\notin \bar{\mathbf{B}}_{k}$. Then there exists $\hat{\bm{\beta}}\lneq \bm{\beta}$ such that $\hat{\bm{\beta}}\in S_k(\alpha)$. Let ${\mathbf{x}^*}\in \opt_k(\hat{\bm{\beta}})$, with ${x}^*_{k}=0$. \color{black}Then by \thref{schaeferoriginal}, $\hat{\bm{\beta}}\in S_{k-1}(\alpha)$, which implies $z_{k-1}(\hat{\bm{\beta}})= \alpha$. However, because $z_{k}({\bm{\beta}})= \alpha$, by \thref{nondecreasingbetweeniterations}, we have $z_{k-1}(\bm{\beta})\leq \alpha$. Because $z_{k-1}$ is nondecreasing, we must have $z_{k-1}(\bm{\beta}) = \alpha$, so that $\bm{\beta}\notin \bar{\mathbf{B}}_{k-1}$, a contradiction.
\end{proof}

\begin{makeshiftResult}{Proposition}{\ref{LSMlinind}}
If $\bm{\beta} \notin\bar{\mathbf{B}}_{k-1}$, $\mathbf{a}_1,\dots,\mathbf{a}_k$ are linearly independent, and $\bbeta$ and $\mathbf{a}_k$ are linearly independent, then for all $t\in \mathbb{Z}_+$ such that $\bm{\beta}+t\mathbf{a}_k\in \mathcal{B}$, $\bm{\beta}+t\mathbf{a}_k \notin \bar{\mathbf{B}}_{k}$.
\end{makeshiftResult}
\begin{proof} 
Suppose $\bbeta + t\ba_{k} \in \barB_{k}.$ Because $\bm{\beta} \not\in \barB_{k-1}$, there exists $\bpi \lneq \bbeta$ such that $z_{k-1}(\bbeta) = z_{k-1}(\bpi)$.  Let $x^{1} \in \opt_{k-1}(\bpi)$, then $(x^{1}_{1},\dots,x^{1}_{k-1},t)^{\top}$ is a feasible solution for IP$_k(\bbeta + t\ba_{k})$. Let $\bx^{2} \in \opt_k(\bbeta + t\ba_{k})$; then by \thref{equal}, we have $\sum\limits_{j = 1}^{k} \ba_{j}x^{2}_{j} = \bbeta + t\ba_{k}$, and because $\ba_{1},\dots\ba_{k}$ are linearly independent, $x^{2}_{k} = t$ and $\sum\limits_{j = 1}^{k-1}\ba_{j}x^{2}_{j} = \bbeta$. 

The vector $(x^{2}_{1},\dots,x^{2}_{k-1})^{\top}$ is feasible for IP$_{k-1}(\bm{\beta}$); hence $\sum\limits_{j = 1}^{k-1} c_{j}x^{2}_{j} \leq \sum\limits_{j = 1}^{k-1}c_{j}x_{j}^{1}$. This implies that $\sum\limits_{j = 1}^{k} c_{j}x^{2}_{j} = \sum\limits_{j = 1}^{k-1}c_{j}x^{2}_{j} + c_{k}t \leq \sum\limits_{j = 1}^{k-1}c_{j}x^{1}_{j} + tc_{k}$. Because $\bx^{2} \in \opt_k(\bbeta + t\ba_{k})$, and $(x^{1}_{1},\dots,x^{1}_{k-1},t)$ is feasible for IP$_k(\bbeta+t\mathbf{a}_k)$, $(x^{1}_{1},\dots,x^{1}_{k-1},t)\in \opt_k(\bbeta + t\ba_{k})$. However, $\sum\limits_{j = 1}^{k-1}\ba_{j}x^{1}_{j} + t\ba_{k} = \bpi + t\ba_{k} \lneq \bbeta + t\ba_{k},$ which contradicts $\bbeta + t\ba_{k} \in \bar{\bB}_{k}$. Hence, $\bbeta + t\ba_{k} \not\in \barB_{k}$.
\end{proof}

\begin{makeshiftResult}{Proposition}{\ref{anotinB}}
Let $\bm{\beta}\in\bar{\mathbf{B}}_k$. If $\mathbf{a}_j\notin \bar{\mathbf{B}}_k$, then for all ${\mathbf{x}^*}\in\opt_k(\bm{\beta})$, ${x}^*_j=0$.
\end{makeshiftResult}
\begin{proof}
Suppose first that ${\mathbf{x}^*} = \mathbf{e}_j$. Then by \thref{equal}, $\mathbf{a}_j \in \bar{\mathbf{B}}_k$, a contradiction.\\
Suppose on the other hand that $\Vert {\mathbf{x}^*}\Vert_1 \geq 2$, with ${x}^*_j \neq 0$. Let $\hat{\mathbf{x}}=\mathbf{e}_j\lneq{\mathbf{x}^*}$. Then by \thref{down}, $\sum_{i=1}^k \mathbf{a}_i\hat{x}_i=\mathbf{a}_j\in \bar{\mathbf{B}}_k$, also a contradiction.
\end{proof}

\begin{makeshiftResult}{Corollary}{\ref{relationship}} 
$\bar{\mathbf{B}}_k\subseteq \left\{\bm{\beta}\in\bar{\mathcal{B}}\,\vert\, \bm{\beta}=\hat{\bm{\beta}}+t\mathbf{a}_k,\hat{\bm{\beta}}\in\bar{\mathbf{B}}_{k-1},t\in\mathbb{Z}_+ \right\}$.
\end{makeshiftResult}
\begin{proof}
Suppose $\bar{\bbeta} \in \barB_{k}$, then \thref{equal} implies that there exists ${\bx^*} \in \opt_k(\bar{\bbeta})$ such that $\slj^{k}\ba_{j}{x}^*_{j} = \bar{\bbeta}.$ By \thref{schaefer}, $\bar{\bbeta} - \ba_{k}{x}^*_{k} \in \barB_{k-1}$; hence, $\bar{\bbeta} \in \left\{\bm{\beta}\in\bar{\mathcal{B}}\,\vert\, \bm{\beta}=\hat{\bm{\beta}}+t\mathbf{a}_k,\hat{\bm{\beta}}\in\bar{\mathbf{B}}_{k-1},t\in\mathbb{Z}_+ \right\}$.
\end{proof}

\begin{makeshiftResult}{Lemma}{\ref{LSMkton}}
Suppose $\ba_{k} \in \bar{\bB}_{k}$ and $z_{k}(\ba_{k}) = c_{k}$, for some $k \in \{1,\dots,n\}$. Further suppose that for all $\hat{k} \in \mathbb{Z}_{+}$ such that $k < \hat{k} \leq n$, we have $\ba_{k} - \ba_{\hat{k}} \not\in \cB$. Then, $z_{n}(\ba_{k}) = c_{k}$ and $\ba_{k} \in \bar{\bB}$.
\end{makeshiftResult}
\begin{proof}
Suppose $z_n(\mathbf{a}_k)\neq c_k$. Then $\mathbf{e}_k\notin \opt_n(\mathbf{a}_k)$ and there exists $\mathbf{x}^*\in \opt_n(\mathbf{a}_k)$ such that $\sum_{j=1}^n \mathbf{a}_jx^*_j\leq \mathbf{a}_k,\,x^*_k=0$. Therefore, there exists $\hat{k}\neq k$ such that $x^*_{\hat{k}}\neq 0$, then $\mathbf{a}_{\hat{k}}\leq \mathbf{a}_k$, which contradicts $\mathbf{a}_k-\mathbf{a}_{\hat{k}}\notin \mathcal{B}$.\par
Suppose $\mathbf{a}_k\notin \bar{\mathbf{B}}$. Then there exists $\bm{\pi}\lneq\mathbf{a}_k$ such that $z_n(\bm{\pi})=z_n(\mathbf{a}_k)=c_k$. Since there does not exist $\hat{k} \in \{k+1, ..., n\}$ such that $\mathbf{a}_{\hat{k}} \leq \mathbf{a}_k$, it follows that there does not exist $\hat{k} \in \{k+1, ..., n\}$ such that $\ba_{\hat{k}} \leq \bpi$, and this indicates $z_k(\bm{\pi})=c_k$, which is a contradiction of the assumption that $\mathbf{a}_k\in \bar{\mathbf{B}}_k$.
\end{proof}

\begin{makeshiftResult}{Proposition}{\ref{altvalue}}
Exactly one of the following holds:
\begin{enumerate}[label = (\roman*)]
    \item $z_{k-1}(\mathbf{a}_k) < c_k$. Then $z_{k}(\mathbf{a}_k)= c_k$ and $\mathbf{a}_k\in \bar{\mathbf{B}}_k$.
    \item $z_{k-1}(\mathbf{a}_k) = c_k$. Then $z_{k}(\mathbf{a}_k)= c_k$.
    \item $z_{k-1}(\mathbf{a}_k) > c_k$. Then $z_{k}(\mathbf{a}_k)=z_{k-1}(\mathbf{a}_k)$.
\end{enumerate}
Thus, $z_k(\mathbf{a}_k) = \max\{z_{k-1}(\mathbf{a}_k), c_k\}$. In addition, if $z_{k-1}(\mathbf{a}_k) \geq c_k$, $\mathbf{a}_k\in \bar{\mathbf{B}}_k$ if and only if $\mathbf{a}_k\in \bar{\mathbf{B}}_{k-1}$.
\end{makeshiftResult}
\begin{proof}
\begin{enumerate}[label = (\roman*), align=left, leftmargin=0pt, labelindent=\parindent, listparindent=\parindent, labelwidth=0pt, itemindent=!]
\par
    \item Suppose $z_{k-1}(\mathbf{a}_k) < c_k$. The vector $\mathbf{e}_k$ is feasible for IP$_k(\mathbf{a}_k)$ with value $c_k$. Suppose there exists $\mathbf{x}^* \in \opt_k(\mathbf{a}_k)$ such that $\sum_{j=1}^k c_jx_j^* > c_k$. Then $\mathbf{x}_k^* = 0$, contradicting $z_{k-1}(\mathbf{a}_k) < c_k$. Similarly, there cannot exist $\bbeta \lneq \mathbf{a}_k$ such that $z_k(\bbeta) \geq c_k$, so $\mathbf{a}_k\in \bar{\mathbf{B}}_k$.
    \item Suppose $z_{k-1}(\mathbf{a}_k) = c_k$. Then by analogous reasoning to (i), $z_k(\mathbf{a}_k) = c_k$.
    \item Suppose $z_{k-1}(\mathbf{a}_k) > c_k$. Then by analogous reasoning to (i), $z_k(\mathbf{a}_k) = z_{k-1}(\mathbf{a}_k)$.
\end{enumerate}
\noindent Next assume $z_{k-1}(\mathbf{a}_k) \geq c_k$. If $\mathbf{a}_k \notin \bar{\bB}_{k-1}$, then $\mathbf{a}_k \notin \bar{\bB}_k$. On the other hand, if $\mathbf{a}_k \in \bar{\bB}_{k-1}$, by analogous reasoning to (i), there cannot exist $\bbeta \lneq \mathbf{a}_k$ such that $z_k(\bbeta) \geq z_k(\mathbf{a}_k)$, so $\mathbf{a}_k \in \bar{\bB}_k$.
\vspace{0pt}
\end{proof}

\begin{makeshiftResult}{Lemma}{\ref{easyiso}}
Given $\bbeta_1\in \cB$, $\bbeta_2 \in T(\bbeta_1)$, there exists an isovalue curve $d': [0,1]\to\cB$ from $\bbeta_1$ to $\bbeta_2$ such that $\lfloor d'(\zeta)\rfloor = \lfloor d'(\theta)\rfloor$ implies $\lfloor d'(\zeta)\rfloor = \lfloor d'(\eta)\rfloor$ for all $0\leq \zeta<\eta<\theta\leq 1$  - that is, $\lfloor d'\rfloor$ takes on any given value for at most a single connected subset of $[0,1]$.
\end{makeshiftResult}
\begin{proof}
Given a continuous isovalue curve $d:[0,1]\to\cB$ from $\bbeta_1$ to $\bbeta_2$, we define the set $\mathcal{S}(d)$ as the set of values in $\bar{\cB}$ which are given by $\lfloor d\rfloor$ for multiple, non-connected subsets of $[0,1]$. In particular, $\mathcal{S}(d) = \left\{\bm{\gamma}^1, ..., \bm{\gamma}^r\right\}\subset \bar{\cB}$, so that for all $\bm{\gamma}^i\in \mathcal{S}(d)$, there exist $0\leq \zeta^i<\eta^i<\theta^i\leq 1$ for which $\lfloor d(\zeta^i)\rfloor = \lfloor d(\theta^i)\rfloor = \bm{\gamma}^i\neq \lfloor d(\eta^i)\rfloor$, and so that for any $\bm{\gamma} \in T(\bbeta_1)$ but not in $\mathcal{S}(d)$, for any $0\leq \zeta<\eta<\theta\leq 1$, $\lfloor d(\zeta)\rfloor = \lfloor d(\theta)\rfloor = \bm{\gamma}$ implies $\lfloor d(\eta)\rfloor = \bm{\gamma}$.
Note that since any particular continuous isovalue curve $d$ is continuous and has a bounded domain, each component of $d$ is also bounded, so $\mathcal{S}(d)$ is finite. Let $\bar{d}:[0,1]\to\cB$ be a continuous isovalue curve from $\bbeta_1$ to $\bbeta_2$ such that $\vert \mathcal{S}(\bar{d})\vert$ is minimized. Suppose the statement does not hold; then $\vert \mathcal{S}(\bar{d})\vert > 0$. Without loss of generality, let $\zeta^1$ be in the first connected subset of $[0,1]$ for which $\lfloor \bar{d}\rfloor = \bm{\gamma}^1$ and let $\theta^1$ be in the last connected subset of $[0,1]$ for which $\lfloor \bar{d}\rfloor = \bm{\gamma}^1$. Note however that $\left\{\bbeta'\in \cB\,\vert\,\lfloor \bbeta' \rfloor = \bm{\gamma}^1\right\}$ is a convex set over which $z$ is constant, and $\{\bar{d}(\theta^{1}), \bar{d}(\zeta^{1})\} \subset T(\beta_{1})$, so the line segment from $\bar{d}(\zeta^{1})$ to $\bar{d}(\theta^{1})$ must be a subset of $T(\bbeta_1)$. Thus, consider $\hat{d}:[0,1]\to\cB$:
$$\hat{d}(t)\coloneqq\left\{\begin{array}{ll}
    \bar{d}(t) & t \in [0,1]\setminus [\zeta^1, \theta^1] \\
    \frac{t-\zeta^1}{\theta^1-\zeta^1}\bar{d}(\zeta^1)+\frac{\theta^1-t}{\theta^1-\zeta^1}\bar{d}(\theta^1) & t \in [\zeta^1, \theta^1]
\end{array}\right..$$
Observe that $\hat{d}$ as defined is a continuous isovalue curve from $\bbeta_1$ to $\bbeta_2$ and $\mathcal{S}(\hat{d})$ has at most $\vert \mathcal{S}(\bar{d})\vert-1$ members, a contradiction, so there exists $d'$ for which $\mathcal{S}(d') = \emptyset$.
\end{proof}\color{black}

\begin{makeshiftResult}{Corollary}{\ref{finseq}}
Given ${\bm{\beta}}\in\cB$, $\bm{\beta}_a$ and $\bm{\beta}_b\in T({\bm{\beta}})$, there exists a finite sequence of points,\\$\left\{\bm{\gamma}^0 = \bm{\beta}_a, ..., \bm{\gamma}^r = \bm{\beta}_b\right\}$, such that $\bm{\gamma}^i\in T({\bm{\beta}})$ and $\vert \bm{\gamma}^{i+1}-\bm{\gamma}^i\vert = \lambda_i\mathbf{e}_j$ for all $i$, where $\lambda_i \in (0,1]$ for all $i$ and $\mathbf{e}_j$ is a unit vector for $j \in \{1, ..., m\}$.
\end{makeshiftResult}
\begin{proof}
Note that $\lfloor \bm{\beta}_a\rfloor$ and $\lfloor \bm{\beta}_b\rfloor$ are both in $T(\bm{\beta})$. If $s$ and $t$ are the number of fractional components of $\bm{\beta}_a$ and $\bm{\beta}_b$, respectively, then the first $s$ members of the sequence after $\bm{\beta}_a$ can be defined as the component by component rounding down of $\bm{\beta}_a$, so that $\bm{\gamma}^{s} = \lfloor{\bm{\beta}_a}\rfloor$. The same can be done for the last $t$ members of the sequence, so that $\bm{\gamma}^{r-t} = \lfloor \bm{\beta}_b\rfloor$. These first $s$ and final $t$ terms are all in $T(\bm{\beta})$, and there is a finite number of them. Then by \thref{Csteppro}, there exists a finite sequence from $\bm{\gamma}^{s}$ to $\bm{\gamma}^{r-t}$ in $C(\bm{\beta})$, so prepending the first $s$ terms and appending the final $t$ terms to that sequence will yield the desired result.
\end{proof}

\begin{makeshiftResult}{Proposition}{\ref{LSMcor}}
For any $\bar{\bm{\beta}}\in\cB$, $\bm{\beta}\in T(\bar{\bm{\beta}})$ if and only if $z(\bm{\beta}) = z(\bar{\bm{\beta}})$ and there exists $\hat{\bm{\beta}} \in $ \wenxin{$\bar{\mathbf{B}}$} \color{black} $\cap T(\bar{\bm{\beta}})$ such that $\hat{\bm{\beta}}\leq \bm{\beta}$.
\end{makeshiftResult}
\begin{proof}
By definition of $T(\bbeta)$, if $z(\bbeta) \neq z(\bar{\bbeta})$, then $\bbeta \not\in T(\bar{\bbeta})$. Suppose $z(\bm{\beta}) = z(\bar{\bm{\beta}})$ but there does not exist level-set-minimal $\hat{\bm{\beta}} \leq \bm{\beta}$ in the same MC-level set as $\bar{\bm{\beta}}$. Note that there must exist at least one $\tilde{\bm{\beta}} \leq \bm{\beta}$ which is level-set-minimal and for which $\tilde{\bbeta} \in T(\bbeta)$. However, there is no continuous curve from $\bar{\bm{\beta}}$ to any such $\tilde{\bm{\beta}}$; thus, because there is a continuous curve from $\tilde{\bm{\beta}}$ to $\bm{\beta}$, there cannot be a continuous curve from $\bar{\bm{\beta}}$ to $\bm{\beta}$, i.e., $\bbeta \not\in T(\bar{\bbeta})$. On the other hand, suppose there does exist level-set-minimal $\hat{\bm{\beta}} \leq \bm{\beta}$ in the same MC-level set as $\bar{\bm{\beta}}$. Then there exists a continuous curve from $\bar{\bm{\beta}}$ to $\hat{\bm{\beta}}$, and from $\hat{\bm{\beta}}$ to $\bm{\beta}$, so that there exists a single continuous curve from $\bar{\bm{\beta}}$ to $\bm{\beta}$, and $\bm{\beta} \in T(\bar{\bm{\beta}})$.
\end{proof}

\begin{makeshiftResult}{Corollary}{\ref{LSMseq}}
Let \wenxin{$\bar{\bbeta} \in \bar{\mathbf{B}}$} \color{black} such that there exists $\hat{\bbeta} \in \bar{\mathbf{B}}\cap T(\bar{\bm{\beta}})$, $\hat{\bbeta}\neq \bar{\bbeta}$. Then there exists a sequence $V = \{v_1, ..., v_s\}$ of distinct vectors in $\bar{\mathbf{B}}\cap T(\bar{\bm{\beta}})$ such that $\bar{\bbeta} = v_1$, $\hat{\bbeta} = v_s$, and for all $i \in \{1, ..., s-1\}$, there exists $\bbeta_i \in T(\bar{\bbeta})$ such that $v_i\lneq \bbeta_i$ and $v_{i+1}\lneq \bbeta_i$.
\end{makeshiftResult}
\begin{proof}
Let $U = \{u_1, ..., u_t\}$ be the finite sequence of points in $T(\bar{\bbeta})$ from $\bar{\bbeta}$ to $\hat{\bbeta}$ whose existence is guaranteed by \thref{Csteppro}. Set $V = \{\bar{\bbeta}\}$ \color{black} initially, and set $\tilde{\bbeta} = \bar{\bbeta}$. Let $u_i$ be the last member in the ordering of $U$ such that $\tilde{\bbeta}\lneq u_i$. Then $u_{i+1} = u_i - \mathbf{e}_j$ for some component $j$, since each member of $U$ differs from the next by a single unit vector, and we must have $u_{i+1}\lneq u_i$ since we have assumed that $u_i$ is the final vector in the ordering of $U$ for which $\tilde{\bbeta}\lneq u_i$. By \thref{LSMcor}, there must exist $\bbeta'\in \bar{\mathbf{B}}_n \cap T(\bar{\bbeta})$ such that $\bbeta'\leq u_{i+1}$. Further, we also have $\bbeta'\lneq u_i$, so we can append $\bbeta'$ to $V$. Then update $\tilde{\bbeta} = \bbeta'$ and repeat until $\bbeta' = \hat{\bbeta}$ is added to $V$. Since $U$ is finite, the process will complete in finitely many steps.
\end{proof}

\begin{makeshiftResult}{Proposition}{\ref{bk=b}}
If there exists $k\in\{1,\dots,n\}$ such that $\bar{\mathbf{B}}_k=\bar{\mathcal{B}}$, then for all $\hat{k}\geq k,\hat{k}\in \{1,\dots,n\}$, $\bar{\mathbf{B}}_{\hat{k}}=\bar{\mathcal{B}}$.
\end{makeshiftResult}
\begin{proof}
Fix $\hat{k}$ such that $k < \hat{k} \leq n$ (if $k = n$, the result is immediate). Suppose there exists $\bm{\beta} \in \bar{\mathcal{B}}$ such that $\bbeta\notin \bar{\mathbf{B}}_{\hat{k}}$. Then there must exist $\bm{\pi}\lneq\bm{\beta}$ such that $z_{\hat{k}}(\bm{\pi})=z_{\hat{k}}(\bm{\beta})$. By \thref{elementary}, we have
\[z_{\hat{k}}(\bm{\pi})=z_{\hat{k}}(\bm{\beta})=z_{\hat{k}}(\bm{\pi}+\bm{\beta}-\bm{\pi})\geq z_{\hat{k}}(\bm{\pi})+z_{\hat{k}}(\bm{\beta}-\bm{\pi}).\]
Therefore, $z_{\hat{k}}(\bm{\beta}-\bm{\pi})=0$. By \thref{notOptimal}, $0\leq z_{k}(\bm{\beta}-\bm{\pi})\leq z_{\hat{k}}(\bm{\beta}-\bm{\pi})=0$; hence, $\bbeta - \bpi \not\in \bar{\bB}_{k}$, a contradiction. As such, $\bar{\mathbf{B}}_{\hat{k}}=\bar{\mathcal{B}}$.
\end{proof}

\end{document}